\documentclass[10pt]{amsart}
\usepackage{graphicx}
\usepackage[centertags]{amsmath}
\usepackage{amsfonts}
\usepackage{amssymb}
\usepackage{amsthm}
\newtheorem{thm}{Theorem}
\newtheorem{Cor}{Corollary}

\newtheorem{lem}{Lemma}[section]

\newtheorem{prop}[lem]{Proposition}
\theoremstyle{definition}

\theoremstyle{remark}
\newtheorem{rem}{Remark}[section]
\numberwithin{equation}{section}

\newcommand{\norm}[1]{\left\Vert#1\right\Vert}

\newcommand{\pr}[2]{\langle  #1,#2\rangle}

\newcommand{\calA}{\mathcal{A}}

\newcommand{\calF}{\mathcal{F}}

\newcommand{\calH}{\mathcal{H}}

\newcommand{\calJ}{\mathcal{J}}

\newcommand{\calO}{\mathcal{O}}

\newcommand{\calR}{\mathcal{R}}

\newcommand{\bbZ}{\mathbb{Z}}
\newcommand{\Z}{\mathbb{Z}}

\newcommand{\bbR}{\mathbb R}
\newcommand{\R}{\mathbb R}
\newcommand{\bbC}{\mathbb C}
\newcommand{\bbD}{\mathbb D}

\newcommand{\bbN}{\mathbb N}

\newcommand{\bbH}{\mathbb{H}}

\newcommand{\Tr}{ \mbox{tr}}

\newcommand{\SL}{ \mathrm{SL}}

\newcommand{\PSL}{ \mathrm{PSL}}

\newcommand{\Mat}{\mathrm{Mat}}

\newcommand{\vol}{\mathrm{vol}}

\newcommand{\bs}{\backslash}

\newcommand{\frakD}{\mathfrak{D}}

\newcommand{\fraka}{\mathfrak{a}}

\newcommand{\m}{\mathrm m}

\begin{document}
\title[Density Hypothesis]
{The density Hypothesis for irreducible lattices in $\PSL(2,\bbR)^d$}%
\author{Dubi Kelmer}%
\address{Department of Mathematics, Boston College}
\email{kelmer@bc.edu}
\thanks{This work was partially supported by NSF CAREER grant DMS-1651563.}

%\thanks{}%

\date{\today}%
\dedicatory{}%
\commby{}%

\begin{abstract}
We prove the density hypothesis for congruence subgroups of an irreducible uniform lattice  in $\PSL_2(\R)^d$, extending previous results on the spherical density hypothesis to bound multiplicities of non-tempered non-spherical representations. Our bounds are uniform in the level as well as the spectral parameters.  \end{abstract}

\maketitle
\section*{introduction}
Let $G=\PSL_2(\R)^d$ and let $\Gamma\leq G$ denote an  irreducible co-compact lattice. 
By  Margulis's Arithmeticity Theorem \cite{Margulis91} all such lattices are arithmetic and from the classification of arithmetic lattices (see Weil \cite{Weil60}) it follows that any such lattice is commensurable to a lattice, $\Lambda$, derived from a quaternion algebra defined over some number field. The space $\Gamma\bs G$ then has a family of congruence covers  $\Gamma(\fraka)\bs G$ where $\fraka$ denotes an integral ideal in the corresponding number field. Explicitly, if $\Gamma$ is commensurable to $g^{-1}\Lambda g$ for some $g\in G^d$ we define
$$\Gamma(\fraka)=\Gamma\cap g^{-1}\Lambda(\fraka)g,$$
where $\Lambda(\fraka)\leq \Lambda$ denotes the principal congruence subgroup of level $\fraka$. The main result of this paper is a bound for the multiplicities of non-tempered representations appearing in $L^2(\Gamma(\fraka)\bs G)$ extending the results of \cite{FraczykHarcosMagaMilicevic24, Kelmer10gap} to include non-spherical representations. 

When $\Gamma$ itself is a congruence group the Ramanujan-Selberg conjecture implies that  all non trivial irreducible representations appearing in $L^2(\Gamma \bs G)$ are tempered. While the proof of this conjecture is currently out of reach, some progress towards it can be made in terms of effective bounds on the strong spectral gap. Explicitly, for any irreducible unitary representation $\pi$ on some Hilbert space, $\calH$, denote by $p(\pi)$ the infimum over all $p\geq 2$ such that matrix coefficients $\pr{\pi(g)v}{v}$ are in $L^p(G)$ for a dense set of vectors $v\in \calH$.  Then $\pi$ is tempered if and only if $p(\pi)=2$ and the size of $p(\pi)$ gives a way to quantify how close the representation is to being tempered. The best known bounds towards the Ramanujan-Selberg conjecture imply that when $\Gamma$ is a congruence group any non-trivial  representation contained in $L^2(\Gamma \bs G)$ satisfies $p(\pi)\leq 2.56$ \cite{KimSarnak03}. For a general irreducible co-compact lattice $\Gamma$ (that is not known to be a congruence lattice) even this is not known, and the best effective result is proved in \cite{KelmerSarnak09}, showing that for any $\epsilon>0$ all representations appearing in $L^2(\Gamma \bs G)$, except perhaps finitely many exceptions satisfy that $p(\pi)\leq 6+\epsilon$. Finally, for congruence covers $\Gamma(\fraka)$ it was shown in \cite{Kelmer10gap}, that any new representation appearing in $L^2(\Gamma (\fraka)\bs G)$ , not already coming from a smaller cover, satisfies that $p(\pi)\leq 7+\sqrt{17}+\epsilon$ (while for spherical representation the bound can be improved to $6+2\sqrt{2}+\epsilon$).

For many applications, instead of the Ramanujan-Selberg conjecture one can make do with effective bounds on the multiplicities,  $\m(\pi,\Gamma)$, of a non tempered representations occurring in $L^2(\Gamma \bs G)$. More generally, given a bounded family $\frak{B}\subseteq \hat{G}$ of non tempered representations  one is interested in bounding the sum 
\begin{equation}\m(\frak B, \Gamma)=\sum_{\pi\in\frak{B}}\m(\pi,\Gamma).\end{equation}
The {\em{Density Conjecture}} of Sarnak and Xue \cite{SarnakXue91}, which is expected to hold for congruence groups in more general semi-simple groups,  can then be stated as the bound 
$$\m(\frak B,\Gamma)\leq C \vol(\Gamma\bs G)^{\frac{2}{p(\frak B)}+\epsilon},$$
where $p(\frak{B})=\min\{p(\pi):\pi\in \frak{B}\}$ and $C=C(\frak B,\epsilon)$ is some constant that may depend on the set $\frak{B}$ and the parameter $\epsilon$.
In \cite{SarnakXue91} Sarnak and Xue proved this for congruence subgroups of $\PSL_2(\R)$ and $\PSL_2(\bbC)$, and as was pointed out in \cite{Kelmer10gap} their argument extends to congruence subgroups of irreducible lattices in products of these groups.

In order to use such multiplicity bounds in applications when the group $G$ has higher rank, one often needs good control on the dependence of the constant $C(\frak B,\epsilon)$ on the set $\frak B$. In particular,  having an explicit polynomial bound, in terms of the Casimir eigenvalues attached to representations in $\frak B$, is referred to as the {\em Density Hypothesis} (we refer to  \cite{GolubevKamber23,FraczykHarcosMagaMilicevic24,FraczykGorodnikNevo24} for a more detailed discussion of the Density Hypothesis and its applications). 
One such result was already established in \cite{Kelmer10gap} in the case of $G=\PSL_2(\R)^d$ by introducing the spectral parameter $T(\pi)$ defined below in  \eqref{e:Tpi}. Explicitly, for  any irreducible representation $\pi$ of $\PSL_2(\R)^d$, the Casimir operator $\Omega_j$ of $\PSL(2,\R)$ (acting on the $j$th factor) acts by a scalar $\lambda_j(\pi)\in \R$, and the parameter $T(\pi)$ grows like $\prod_{j=1}^d\sqrt{|\lambda_j(\pi)|+1}$.
 For a set $\frak B$ we define
$ T(\frak B)=\max\{T(\pi): \pi \in \frak B\}.$
 With these parameters, it was shown in \cite{Kelmer10gap} that for $\frak B=\{\pi\}$ a single non-tempered representation of $\PSL_2(\R)^d$, and any family of congruence covers $\Gamma(\frak a)$ of some fixed arithmetic irreducible lattice $\Gamma$, one can take 
 $C(\frak B,\epsilon)=C(\Gamma,\epsilon) T(\pi)^{\frac{1}{2}+\frac{1}{p(\pi)}+\epsilon}$,
 and that if $\pi$ is a spherical representation one can replace $\frac{1}{2}+\frac{1}{p(\pi)}$ by $\frac{2}{p(\pi)}$.
 Recently, this result was extended in \cite{FraczykHarcosMagaMilicevic24} to deal with more general families of irreducible co-compact lattices in products of $\PSL_2(\R)$ and $\PSL_2(\bbC)$, as well as more general sets $\frak B$ of {\em{spherical}} representations. 
The goal of this note is to improve the results of \cite{Kelmer10gap}, dealing with congruence covers of a fixed lattice $\Gamma$, 
so that the general bound for non spherical representations is as good as the bound for spherical representations. Moreover, we will consider general families $\frak B$ of non tempered representations. Explicitly we show the following.

\begin{thm}\label{t:main1}
For $G=\PSL_2(\R)^d$, given any uniform irreducible lattice $\Gamma\leq G$  and $\epsilon>0$ there is a constant $C=C(\Gamma,\epsilon)$ depending only on $\Gamma$ and $\epsilon$ such that for any family of congruence subgroups $\Gamma(\fraka)\leq \Gamma$  for any  $\pi\in 
\hat{G}$  
$$ \m(\pi,\Gamma(\fraka))\leq C \cdot  (T(\pi)\vol(\Gamma(\frak a)\bs G))^{\frac{2}{p(\pi)}+\epsilon}$$
\end{thm}

Combining this result with the argument of \cite{Kelmer10gap} we get the following consequence regarding the strong spectral gap for congruence covers.
\begin{Cor}
Given a uniform irreducible lattice $\Gamma\leq \PSL_2(\R)^d$ there is an ideal  $\frak{d}$ in the corresponding number field, such that for any $\epsilon>0$ all ideals $\frak{a}$ prime to $\frak{d}$ with sufficiently large norm, all  new representations in $L^2(\Gamma(\frak a)\bs \PSL_2(\R)^d)$ satisfy that $p(\pi)<6+2\sqrt{2}+\epsilon$.
\end{Cor}
For some application one needs to control $\m(\frak B,\Gamma)$ for general bounded sets of representations. In this case we have the following. 
\begin{thm}\label{t:main2}
For $G=\PSL_2(\R)^d$, given any uniform irreducible lattice $\Gamma\leq G$  and $\epsilon>0$ there is a constant $C=C(\Gamma,\epsilon)$ depending only on $\Gamma$ and $\epsilon$ such that for any family of congruence subgroups $\Gamma(\fraka)\leq \Gamma$,
for any bounded set $\frak{B}\subseteq \hat{G}$ of non tempered representations 
$$
\m(\frak B,\Gamma) \leq C\cdot  (T(\frak{B})^2\vol(\Gamma(\frak a)\bs G))^{\frac{2}{p(\frak{B})}+\epsilon}.
$$
\end{thm}
\begin{rem}
Note that taking $\frak B=\{\pi\}$ in Theorem \ref{t:main2} the bound we get is not as good as the one in Theorem \ref{t:main1}.
On the other hand if $\frak{B}$ is a large set, say $\frak{B}=\{\pi: T(\pi)\leq T_0\}$ then the bound in Theorem \ref{t:main2} is much better than the bound we would get from summing $\sum_{\pi\in \frak{B}} \m(\pi,\Gamma(\fraka))$ and using Theorem \ref{t:main1} for each term separately. 
\end{rem}

\section{Preliminaries and Notation}\label{s:BG}
\subsection{Notations}
We write $X\ll Y$ or $X=O(Y)$ to indicate that $X\leq CY$ for some constant $C$.
If we wish to emphasize that this constant depends on some parameters, we indicate it by subscripts, for example $X\ll_\epsilon Y$. We will write $X\asymp Y$ to indicate that $X\ll Y\ll X$.

\subsection{Coordinates}
Throughout this note we will denote by $G=\PSL_2(\bbR)$ and let $P,A,K\subset G$ denote the subgroups of upper triangular matrices, diagonal matrices, and orthogonal matrices respectively. For $t\in \bbR$ and $\theta\in[-\pi,\pi]$ let $a_t=\left(\begin{smallmatrix}e^{t/2} & 0\\0& e^{-t/2}\end{smallmatrix}\right)\in A$ and $k_\theta=\left(\begin{smallmatrix}\cos(\frac{\theta}{2}) & \sin(\frac{\theta}{2}) \\-\sin(\frac{\theta}{2})& \cos(\frac{\theta}{2})\end{smallmatrix}\right)\in K$. For $z=x+iy\in \bbH$ let $p_z=\left(\begin{smallmatrix}y^{1/2} & xt^{-1/2} \\ 0 & y^{-1/2}\end{smallmatrix}\right)\in P$.
The decomposition $G=PK$ gives us the coordinate system $g(z,\theta)=p_zk_\theta$. The Haar measure in these coordinates is given by $dg(z,k)=d\mu(z)dk$ where $d\mu(z)=\tfrac{dxdy}{y^2}$ and $dk=\frac{d\theta}{2\pi}$. We also have the decomposition $G=KA^+K$, where $A^+=\{a_t: t\geq 0\}$, and in the corresponding coordinates the Haar measure is given by $dg(k,t,k')=2\pi \sinh(t)dtdkdk'$.

For $g=\left(\begin{smallmatrix} a& b\\ c&d\end{smallmatrix}\right)\in G$ let $\norm{g}^2=\Tr(g^tg)=|a|^2+|b|^2+|c|^2+|d|^2$ and $H(g)=d(gi,i)$ where $d(\cdot,\cdot)$ denotes the hyperbolic distance on $\bbH$. We thus have that $H(ka_tk')=t$ and a simple calculation shows that $\norm{g}^2=2\cosh(H(g))$.
The triangle inequality for the hyperbolic distance implies that for any $g,\gamma\in G$
\begin{equation}\label{e:triangle}
|H(g^{-1}\gamma g)-H(\gamma)|\leq 2H(g),
\end{equation}
For any $c>0$ we denote by $B_c\subset G$ the norm ball 
\begin{equation}\label{e:B_c}
B_c=\{g\in G|H(g)\leq c\}=\{g\in G\;|\;\|g\|^2\leq 2\cosh(c)\}
\end{equation}

We record here a simple identity relating the  $KA^+K$ coordinates of conjugates of $k_\theta\in K$.
\begin{lem}\label{l:kak}
For any $t>0$ and $\theta\in [-\pi ,\pi]$ decompose 
$$a_{-t}k_\theta a_t=k_\alpha a_H k_\beta$$
Then
\begin{equation}\label{e:aka1}
\cosh(H)=1+\sin^2(\tfrac{\theta}{2})(\cosh(t)-1),
\end{equation}
and
\begin{equation}\label{e:aka2}
\alpha+\beta=\theta+2\mathrm{Arg}(1+\sinh^2(\frac{t}{2})(1-e^{-i\theta}))
\end{equation}
\end{lem}
\begin{proof}
For the first identity we have that 
$$2\cosh(H)=\|a_{-t}k_\theta a_t\|^2=2\cos^2(\frac{\theta}{2})+2\cosh(t)\sin^2(\frac{\theta}{2}),$$
hence
$$\cosh(H)=1+\sin^2(\tfrac{\theta}{2})(\cosh(t)-1).$$
For the second identity let $w=\frac{1}{\sqrt{2}}\begin{pmatrix}1& i\\i &-1\end{pmatrix}$ denote the transformation sending upper half space, $\bbH$, to the unit disk, $\bbD=\{z\in \bbC: |z|<1\}$, with $w(i)=0$ and let $\tilde{a}_t=wa_t w^{-1}=\begin{pmatrix} \cosh(\frac{t}{2}) & \sinh(\frac{t}{2})\\ \sinh(\frac{t}{2})& \cosh(\frac{t}{2})\end{pmatrix}$ and $\tilde k_\theta=w k_\theta w^{-1}=\begin{pmatrix} e^{i\theta/2} & 0\\ 0& e^{-i\theta/2}\end{pmatrix}$. Writing $\tilde{a}_{-t}\tilde{k}_\theta\tilde{a}_t=\tilde{k}_\alpha \tilde{a}_H\tilde{k}_\beta$ and acting on $0\in \bbD$ we see that 
$$e^{i\alpha}\tanh(\frac{H}{2})=\frac{\cosh(\frac{t}{2})\sinh(\frac{t}{2})(e^{i\theta}-1)}{1+\sinh^2(\frac{t}{2})(1-e^{i\theta})}.$$
Taking inverses we also have that 
$\tilde{a}_{-t}\tilde{k}_{-\theta}\tilde{a}_t=\tilde{k}_{-\beta} \tilde{a}_{-H}\tilde{k}_{-\alpha}$ hence
$$-e^{-i\beta}\tanh(\frac{H}{2})=\frac{\cosh(\frac{t}{2})\sinh(\frac{t}{2})(e^{-i\theta}-1)}{1+\sinh^2(\frac{t}{2})(1-e^{-i\theta})}.$$
Dividing the two identities we get
$$-e^{i(\alpha+\beta)}=-e^{i\theta}\frac{1+\sinh^2(\frac{t}{2})(1-e^{-i\theta})}{1+\sinh^2(\frac{t}{2})(1-e^{i\theta})}.$$
Hence $\alpha+\beta=\theta+2\mathrm{Arg}(1+\sinh^2(\frac{t}{2})(1-e^{-i\theta}))$ as claimed
\end{proof}

\subsection{Unitary Dual}
Denote by $\hat{G}$ the unitary dual of $G$, that we parameterize by $(0,\frac{1}{2})\cup \{\frac{1}{2}+i \bbR^+\}\cup (\bbZ\setminus\{0\})$.
In particular, the spherical irreducible representations of $G$ are given by the principal
series representations $\pi_{s},\;s\in\frac{1}{2}+i\bbR^+$  and the complementary series representations
$\pi_s,\;s\in (0,\tfrac{1}{2})$. We note that the Casimir eigenvalue for $\pi_s$ is $s(1-s)$. The non-spherical representation are given by the discrete series $\frakD_{m},\;m\in\bbZ\setminus\{0\}$ for which the Casimir eigenvalue is $|m|(1-|m|)$. 
The discrete and principal series are both tempered while the complementary series is non-tempered with $p(\pi_s)=\frac{1}{\Re(s)}$. 

An irreducible unitary representation of $G^d$ is of the form $\pi\cong \pi_1\otimes\cdots\otimes\pi_d$ with each $\pi_j\in \hat{G}$. For such a representation $p(\pi)=\max_{j}p(\pi_j)$ and $T(\pi)=\prod_j T(\pi_j)$ where we define the parameters
\begin{equation}\label{e:Tpi}
T(\pi_j)=\left\lbrace\begin{array}{cc}
1 & \pi_j\cong \pi_s: s\in (0,1/2)\\
1+|r| & \pi_j\cong \pi_{\frac{1}{2}+ir}\\
|m| & \pi_j\cong \frak{D}_m
\end{array}\right.
\end{equation}
Note that if $\lambda(\pi_j)$ is the Casimir eigenvalue for $\pi_j$   then $T(\pi_j)\asymp \sqrt{|\lambda(\pi_j)|+1}$.

\subsection{Spherical transforms}
For $m\in \bbZ$ let $\chi_m$ denote the character of $K$ given by $\chi_m(k_\theta)=e^{im\theta}$.
We denote by $\calF_m\subset C^{\infty}(G)$ the subspace of smooth functions satisfying that 
\[\forall\; k,k'\in K,\quad f(kgk')=\chi_m(kk')f(g).\]
The $m$-spherical transform on $\calF_m$ is defined by the integral
\[S_mf(r)=\int_{G}f(g)\phi_{-m,\frac{1}{2}+ir}(g)dg\]
(when it converges), where $\phi_{m,s}\in \calF_m$ is the unique eigenfunction of the Casimir operator with eigenvalue $s(1-s)$ satisfying $\phi_{m,s}(1)=1$.

For $f_1,f_2\in\calF_m$ their convolution $f_1*f_2\in\calF_m$ is given by
\[f_1*f_2(x)=\int_Gf_1(g)f_2(g^{-1}x)dx.\]
Under convolution the spherical transform satisfies
$$S_m(f_1*f_2)(r)=S_mf_1(r)S_mf_2(r).$$
Denote by $\check{f}(g)=\overline{f(g^{-1})}$ then $S_m(\check{f})(r)=\overline{S_mf(\bar{r})}$, in particular for any $f\in\calF_m$ we have that $S_m(f*\check{f})(r)=|S_mf(r)|^2$ is nonnegative for  $r\in \bbR\cup i\bbR$.

For a function $f\in C^\infty_c(G)$ that is a linear combination of functions with different $K$-types, $f=\sum_m f_m$, with each $f_m\in \calF_m$ we define the spherical transform accordingly as $Sf=\sum_m S_mf_m$. In particular, for $f\in \calF_m$ we have that $Sf=S_mf$ and we may omit the subscript from the notation.

When $f\in\calF_0$ is sepherical and is supported on $B_c$, its spherical transform $h(r)=Sf(r)$ is an even holomorphic function of uniform exponential type $c$, or equivalently, its Fourier transform $\hat{h}$ is even, smooth, and supported on $[-c,c]$ (see  \cite[Ch. 5]{Lang85}). Moreover, the spherical transform gives a bijection between the space of smooth compactly supported functions in $\calF_0$ and the space $\rm{PW}(\bbC)$ of even holomorphic functions of uniform exponential type. The inverse transform is given by the following formula (see \cite[Ch 5]{Lang85}) 
\begin{equation}\label{e:Sinv}
S_0^{-1}h(g)=\int_\R h(r)\phi_{\tfrac12+ir}(g)r\tanh(\pi r)dr.
\end{equation}

For $|m|>1$ let $\Phi_m=\phi_{m,|m|}$ denote the spherical function of weight $m$ and eigenvalue $|m|(1-|m|)$. This function is given by the formula
\begin{equation}\label{e:mSpherical}
\Phi_m(ka_tk')=\frac{\chi_m(kk')}{(\cosh\frac{t}{2})^{2|m|}}.
\end{equation}
By orthogonality relations, its spherical transform vanishes unless  unless $r=\pm i(|m|-\tfrac{1}{2})$ and when $|m|>1$ we have
\begin{equation}\label{e:SPhim}
S_m\Phi_m(r)=\left\lbrace\begin{array}{cc}
\frac{4\pi}{2|m|-1} & r=\pm i(|m|-\tfrac{1}{2})\\
0 & \mbox{ otherwise}.
\end{array}\right.
\end{equation}
Note that when $|m|>1$ this function decays sufficiently fast so that the integral defining the spherical transform absolutely converges.

\subsection{Latticed derived from quaternion algebras}
Let $L$ denote a totally real number field of degree $n\geq d$ and denote by $\iota_j:L\to \R$ the different embeddings of $L$ in $\R$.
Given non zero $a,b\in L$ the quaternion algebra  $\calA=\left(\frac{a,b}{L}\right)$  is defined as
$$\calA=\{x+yI+wJ+zK: x,y,z,w\in L\},$$ where $I^2=a, J^2=b$ and $IJ=-JI=K$. We assume that $\calA$ is a division algebra satisfying that $\calA\otimes_{\iota_j}\bbR\cong \Mat_2(\R)$ for $j\leq d$  and $\calA\otimes_{\iota_j}\bbR\cong \Mat_2(\bbC)$  for $d<j\leq n$.  For any $j\in \{1,\ldots, n\}$ there is a natural homomorphism, we still denote by $\iota_j$,  from $\calA$ to $\Mat_2(\bbR)$ (resp. $\Mat_2(\bbC)$) so that  $\iota_j(I)=\begin{pmatrix} \sqrt{\iota_j(a)} & 0\\0& -\sqrt{\iota_j( a)}\end{pmatrix}$ and  $\iota_j(J)=\begin{pmatrix}0 & \iota_j(b)\\1 & 0 \end{pmatrix}$. Let $n_\calA:\calA\to L$ denote the norm map $n_\calA(x+yI+wJ+zK)=x^2-ay^2-bz^2+abz^2$ and note that $\iota_j(n_\calA(\alpha))=\det(\iota_j(\alpha))$.

Let $\calO_L$ denote the ring of integers of $L$ and consider the order $\calR\subset\calA$ defined by
$$\calR=\{x+yI+wJ+zK:x,y,w,z\in \calO_L\}.$$
Let  $\calR^1=\{\alpha\in \calR: n_\calA(\alpha)=1\}$ denote the subgroup of norm one elements and note that for any $1\leq j\leq d$ we have that $\iota_j(\calR^1)\subseteq \PSL_2(\R)$.  Let $\iota=(\iota_1,\ldots, \iota_d)$ and note that 
$\Lambda=\{\iota(\alpha): \alpha\in \calR^1\}$, 
is an irreducible co-compact lattice in $\PSL_2(\R)^d$.  For any integral ideal $\fraka\subseteq \calR$ the corresponding principal congruence subgroup is given by 
$$\Lambda(\frak a)=\{\iota(\alpha): \alpha\in \calR^1,\; \alpha- 1\in \frak a\}.$$

As noted above, any irreducible co-compact lattice $\Gamma$  in $\PSL_2(\R)^d$ is commensurable to a lattice $\Lambda$ derived from a quaternion algebra in this way. That is, there is some $g\in \PSL_2(\R)^d$ such that $\Gamma\cap g^{-1}\Lambda g$ is of finite index in $\Gamma$ and in $g^{-1}\Lambda g$. In this case we define the congruence subgroups of $\Gamma$ by $\Gamma(\frak a)=\Gamma\cap g^{-1}\Lambda(\frak a)g$.

\subsection{Lattice point counting}
Given an irreducible lattice  $\Gamma\leq G^d$ we consider the following counting functions. 
For a partition $\{1,2,\ldots,d\}=J_0\cup J_1\cup J_2\cup J_3$ with $|J_0|=1$ and parameters  $x\geq 1$ large,  $k_j\in \bbN$ for $j\in J_1$, and   $\eta_j\in (0,1)$ for each $j\in J_2$ as well as some uniform constant $C\geq 1$ we consider
\begin{equation}\label{e:NGamma}
N_{\Gamma}(x,\eta,k)=\#\left\{\begin{array}{c}\gamma\in \Gamma\\ \gamma\neq I \end{array}: \begin{array}{cc} \|\gamma_j\|^2\leq x & j\in J_0\\
 \|\gamma_j\|^2\in [k,k+1),\; |\Tr(\gamma_j)|<2 &  j\in J_1\\
 \|\gamma_j\|\leq C,\;| |\Tr(\gamma_j)|-2|\leq \eta_j & j\in J_2\\
 \|\gamma_j\|\leq C & j\in J_3
 \end{array}\right\}.
 \end{equation}
If $J_1$ (or $J_2$) is empty we ommit the parameters $k$ (or $\eta$) accordingly. 
 
We note that if  $\Gamma\leq G^d$  is of finite index in $\Lambda$ then
  $N_{\Gamma(\frak{a})}(x,\eta,k)\leq N_{\Lambda(\frak{a})}(x,\eta,k)$. Hence, the estimates in \cite[Proposition 2.3]{Kelmer10gap} imply the following.
\begin{prop}\label{p:NGamma}
For any family of congruence subgroups $\Gamma(\frak a)\leq \Gamma$ we have
$$N_{\Gamma(\frak{a})}(x,\eta,k)\ll_{\Gamma,C,\epsilon} |\eta|\left( \frac{x^{1+\epsilon}}{V(\frak{a})}+\frac{|k|^\epsilon x^{1/2+\epsilon}}{V(\frak{a})^{2/3}} \right)$$
 where $\eta=\prod_{j\in J_2} \eta_j,\; |k|=\prod_{j\in J_1}k_j$ and $V(\frak{a})=\vol(\Gamma(\frak{a})\backslash  G^d)\asymp [\Gamma:\Gamma(\frak a)]$.
\end{prop}

\section{The Selberg trace formula}
Let $\Gamma$ be an irreducible co-compact lattice in $G^d$.
The space $L^2(\Gamma\bs G^d)$ is the space of Lebesgue
measurable functions on $G$ satisfying that $f(\gamma
g)=f(g)$ and that $\int_{\Gamma\bs G^d}|f(g)|^2dg<
\infty$. The Selberg trace formula relates the spectral decomposition of
$L^2(\Gamma\bs G)$ to the conjugacy classes in $\Gamma$. We
refer to \cite[Sections 1-6]{Efrat87}, \cite[Chapter 3]{Hejhal76}
and \cite{Selberg95} for the full derivation of the trace formula in
this setting. We recall here some basic facts and notations.

\subsection{Trace formula}
Fix a weight $m=(m_1,\ldots,m_d)\in \bbZ^d$ and let $\chi_m(k)=\chi_{m_1}(k_1)\cdots\chi_{m_d}(k_d)$ denote the corresponding character of $K^d$. Let $L^2(\Gamma\bs G^d,m)$ denote the subspace of functions in $L^2(\Gamma\bs G^d)$  satisfying
\[\psi(gk)=\chi_m(k)\psi(g),\; \forall\; k\in K^d.\]
For $\lambda=(\lambda_1,\ldots,\lambda_d)\in \bbR^d$ let $V_\lambda(\Gamma,m)\subset L^2(\Gamma\bs G^d,m)$ denote the eigenspace
\[V_\lambda(\Gamma,m)=\{\psi\in L^2(\Gamma\bs G^d,m)|\Omega_j\psi+\lambda_j\psi=0\},\]
where $\Omega_j$ denotes the Casimir operator of $G$ acting on the $j$'th factor.

For $j=1,\ldots,d$ let $F^{(j)}\in\calF_{m_j}$ with $h_j=SF^{(j)}$ its spherical transform and let $F(g)=\prod_jF^{(j)}(g_j)$ and $h(r)=\prod_jh_j(r_j)$. Taking trace in $L^2(\Gamma\bs G^d,m)$ of the kernel
$$F_\Gamma(g,g')=\sum_{\gamma\in \Gamma}F(g^{-1}\gamma g'),$$
yields the trace formula
\begin{eqnarray}\label{e:trace}
\lefteqn{\sum_k \dim(V_{\lambda_k}(\Gamma,m))h(r_k)=\sum_{\gamma}\int_{\Gamma\bs G^d}F(g^{-1}\gamma g)dg=}\\
\nonumber &&=\vol(\Gamma\bs G^d)F(1)+\sum_{\{\gamma\}}\vol(\Gamma_\gamma\bs G^d_\gamma)\int_{G_\gamma^d\bs G^d} F(g^{-1}\gamma g)dg.
\end{eqnarray}
where the sum on the left is over the set of all eigenvalues $\lambda_k=\frac{1}{4}+r_{k}^2$ of eigenfunctions in $L^2(\Gamma\bs G^d,m)$, and the sum in the second line on the right is over all non trivial conjugacy classes in $\Gamma$, and $G_\gamma$ and $\Gamma_\gamma$ are the centralizers of $\gamma$ in $G$ and $\Gamma$ respectively. 

\subsection{Specialization}
To a representation $\pi=\pi_1\otimes\cdots \otimes\pi_d$ we attache the eigenvalue $\lambda=\lambda(\pi)\in\bbR^d$ such that the Casimir operator of $G$ acts on $\pi_j$ via multiplication by $\lambda_j$, specifically, $\lambda_j=s_j(1-s_j)$ when $\pi_j\cong \pi_{s_j}$ and $\lambda_j=|m_j|(1-|m_j|)$ when $\pi_j=\frakD_{m_j}$.
Let $\tilde{m}\in \bbZ^d$ such that $\tilde{m}_j\geq m_j$ when $m_j>0$ (and respectively $\tilde{m}_j\leq m_j$ when  $m_j<0$). Since in any irreducible representation $\pi_j$ there is a unique (up to scaling) vector of $K$-type $\tilde{m}_j$ we have that
$$\m(\pi,\Gamma)=\dim(V_\lambda(\Gamma,\tilde{m})).$$

In order to estimate $\m(\pi,\Gamma)$ when $\pi$ is non-spherical in some factors we can use this trace formula making sure that the $|\tilde{m}_j|$ are sufficiently large. This can be done to handle non spherical representations when all $m_j$ are bounded. 
 However, when some $m_j$ are large this is rather wasteful and does not give good control on the dependence on $m$. To overcome this we specialize our test functions to pick up specifically non-spherical representations of a prescribed type.

Fixing a partition $\{1,\ldots, d\}=J_1\cup J_2\cup J_3$ we may assume that $\pi=\pi_1\otimes\cdots\otimes \pi_d$ with $\pi_j$ spherical for $j\in J_1$ and $\pi_j\cong \frakD_{m_j}$ for $j\in J_2\cup J_3$ with $|m|_j=1$ for $j\in J_2$ and $|m_j|>1$ when $j\in J_3$. Set $m_j=0$ for $j\in J_1$ and let $m=(m_1,\ldots,m_d)\in\bbZ^d$. Now, for $j\in J_1\cup J_2$ take $F^{(j)}\in \calF_{m
_j}$ compactly supported and let $h_j=SF^{(j)}$ its spherical transform, while for $j\in J_3$ we let $F^{(j)}=\frac{2|m_j|-1}{4\pi}\Phi_{m_j}$ with $\Phi_{m_j}$ denoting the $m_j$-spherical function defined in \eqref{e:mSpherical}. With this choice of test function the  trace formula becomes 
\[\sum_{k} \m(\pi_k,\Gamma)\prod_{j\in J_1\cup J_2}h_j(r_{k,j})=\sum_{\gamma\in \Gamma}\int_{\Gamma\bs G^d}F(g^{-1}\gamma g)dg,\]
where the sum is over all representation $\pi_k=\pi_{k,1}\otimes\cdots\otimes\pi_{k,d}$ such that  for $j\in J_1$, $\pi_{k,j}\cong \pi_{s}$ is spherical with parameter $s=\frac{1}{2}+ir_{k,j}$; for $j\in J_2$, either $\pi_{k,j}\cong \pi_{s}$ is spherical with parameter $s=\frac{1}{2}+ir_{k,j}$ or $\pi_{k,j}\cong \frakD_{m_j}$ and $r_{k,j}=\frac{i}{2}$;  and for $j\in J_3$ we have $\pi_{k,j}\cong\frakD_{m_j}$.
Further collecting the sum into $\Gamma$-conjugacy classes we can rewrite the right hand side of the trace formula as 
$$\vol(\Gamma\bs G)F(1)+\sum_{\{\gamma\}}\vol(\Gamma_\gamma\bs G^d_\gamma)\int_{G_\gamma^d\bs G^d} F(g^{-1}\gamma g)dg,$$
where 
$$\int_{G_\gamma^d\bs G^d} F(g^{-1}\gamma g)dg=\prod_j\int_{G_{\gamma_j}\bs G} F^{(j)}(g^{-1}\gamma_j g)dg.$$
These orbital integrals can be explicitly computed in terms of $h_j(r)$ (see e.g. \cite[Section 3.3]{Kelmer10gap}). 
We only note here that for $j\in J_3$ the orbital integrals
 \begin{equation}\label{e:Htransform2}
 \int_{G_{\gamma_j}\bs G} F^{(j)}(g^{-1}\gamma_j g)dg=0,
 \end{equation}
 vanishes unless $|\Tr(\gamma_j)|<2$, and hence, the only contribution to the trace formula comes from lattice points $\gamma\in \Gamma$ satisfying that $|\Tr(\gamma_j)|<2$ for all $j\in J_3$.

\section{Construction of test functions}
In this section we construct some explicit families of test functions and give estimates on their spherical transforms and orbital integrals.
\subsection{Test function localized on exceptional spectrum}
In order to bound the multiplicities of non-tempered representation we will use the following test function in the trace formula whose spherical transform is large on the exceptional spectrum. 
\begin{prop}\label{p:FR}
Given a large parameter $R\geq 1$ let $f_R\in \calF_0$ be smooth, supported on $B_R$ take values in $[0,1]$ and satisfy that $f(g)=1$  on $B_{R-1}$. Then the test function $F^{ex}_R=f_R*\check{f}_R$ satisfies the following properties:
\begin{enumerate}
\item The spherical transform $SF^{ex}_R(r)\geq 0$ for $r\in \
\R\cup i\R$, and for $r\in (0,1/2)$ we have
$SF^{ ex}_R(ir)\asymp e^{2R(\tfrac{1}{2}+r)}$.

\item $F^{ ex}_R$ is supported on $B_{2R}$ and satisfies that
$|F^{ex}_R(g)|\ll \frac{e^R}{\|g\|}$. In particular,  for any $c>0$ there is $C\geq 1$ such that for all $\gamma\in \PSL_2(\R)$ we have 
$$\int_{B_c}F^{ex}_R(g^{-1}\gamma g)dg\ll_c\left\lbrace\begin{array}{cc}  \frac{e^R}{\|\gamma\|}& \|\gamma\|\leq C e^R\\
0 & \mbox{ otherwise.}
\end{array}\right..$$
\end{enumerate}
\end{prop}
\begin{proof}
The first part follows from the fact that for $s\in (0,1/2)$ the spherical function 
$\phi_s(g)$ is real,  non negative and satisfies $\phi_s(ka_tk')=\phi_s(a_t)\asymp e^{-st}$ as $t\to \infty$, and that $SF_R^{ex}(ir)=|Sf_R(ir)|^2$.
The second part follows from  \cite[Lemma 4.1 ]{Kelmer10gap}.
\end{proof}

\subsection{Test function localized on large spectral window}
Next we construct a test function whose spherical transform is localized on the spectral window $|r|\leq T$, and prove some estimates on its orbital integrals. We start with a fixed smooth bi-$K$ invariant function $f$ supported on $B_1$ and let $F(g)=f*\hat{f}$ and $h(r)=SF(\frac12+ir)$. Then $h(r)$ is holomorphic of uniform exponential type and it is nonnegative for $r\in \bbR\cup i \bbR$. For a large parameter $T\geq 1$ let $h_T(r)=h(\frac{r}{T})$ and let  $F^{sp}_T$ denote the inverse spherical transform $F^{sp}_T=S_0^{-1}h_T$. 
\begin{prop}\label{p:FT2}
The functions $F^{sp}_T$ is bi-$K$ invariant, supported on $B_{2/T}$ and satisfies $|F^{sp}_T(g)|\ll T^2$ there.  Moreover, it satisfies the following.
\begin{enumerate}
\item Its spherical transform $h_T(r)\geq 0$ for $r\in \R\cup i\R$ and  $h_T(r)\gg1$ for $|r|\leq T$.
\item  For any $c>0$ there is a constant $C\geq 1$  such that for $\gamma\in \PSL_2(\R)$ and any $\eta>\frac{4}{T^2}$ we can bound:
$$\int_{B_c} |F^{sp}_T(g^{-1}\gamma g)|dg\ll_c \left\lbrace\begin{array}{cc}
0 & \|\gamma \|\geq C \\
0 & |\Tr(\gamma)|\geq 2+\frac{4}{T^2}\\
1/\eta & \|\gamma\|\leq  C \mbox{ and } |\Tr(\gamma)| \leq 2-\eta\\
T^2  &  \|\gamma\|\leq  C \mbox{ and }| |\Tr(\gamma)|-2| \leq \frac{4}{T^2}\\
\end{array}\right.
$$
\end{enumerate}
\end{prop}
\begin{proof}
Since $F$ is supported on $B_2$, the Fourier transform $\hat{h}$ of $h=SF$ is supported on $[-2,2]$ and hence $\hat{h}_T(u)=T\hat{h}(Tu)$ is supported on $[\frac{-2}{T},\frac{2}{T}]$ and $F^{sp}_T$ is supported on $B_{\frac{2}{T}}$. Using the inverse transform formula we can also conclude that 
$$F^{sp}_T(g)\ll \int_\bbR |h_T(r)|r\tanh(\pi r)dr\ll T^2.$$
Since for any $g\in B_c$ we have that $H(g^{-1}\gamma  g)\geq H(\gamma)-2c$, then as long as $H(\gamma)\geq 2c+2$ we have that $H(g^{-1}\gamma  g)\geq 2\geq \frac{2}{T}$. Hence 
$F^{sp}_T(g^{-1}\gamma  g)=0$ for all $\gamma$ with $\|\gamma\|\geq C$ (where $C\geq 2\cosh(2+2c)$). Also note that $\Tr(\gamma)=\Tr(g^{-1}\gamma  g)$, and writing 
$g^{-1}\gamma  g=ka_Hk'$  we also have that $\Tr(\gamma)=\Tr(k'ka_H)=2\cos(\theta/2)\cosh(H)$ (where $k'k=k_\theta$). Since $F^{sp}_T(g^{-1}\gamma  g)=0$ unless $H\leq  \frac{2}{T}$ we have that $\int_{B_c}F^{sp}_T(g^{-1}\gamma  g)dg=0$ unless  $|\Tr(\gamma)|\leq 2+\frac{4}{T^2}$. 

We are now left with dealing with the case of $\|\gamma\|\leq C$ and $|\Tr(\gamma)|\leq 2-\eta$ for $\eta\geq \frac{4}{T^2}$. In this case $\gamma$ is elliptic so we can write 
$\gamma=\tau^{-1}k_\theta \tau$ for some $\tau\in \SL_2(\bbR)$. Writing $g=\tau^{-1}k_1 a_t k_2$ we have that
$F^{sp}_T(g^{-1}\gamma  g)=F_T(a_{-t}k_\theta a_t)$. Recall that if $H=H(a_{-t}k_\theta a_t)$ then $\cosh(H)-1=\sin^2(\frac \theta 2)(\cosh(t)-1)$ and hence 
$F^{sp}_T(a_{-t}k_\theta a_t)=0$ when $\sin^2(\theta/2)(\cosh(t)-1)\geq \frac{4}{T^2}$ or equivalently when $t\leq \frac{2}{T |\sin(\theta/2)|}$. We can thus bound 
$$\int_B F^{sp}_T(g^{-1}\gamma g)dg\leq T^2 \int_0^{\frac{1}{T |\sin(\theta/2)}|}\sinh(t)dt\ll \frac{1}{\sin^2(\theta/2)}=\frac{1}{1-\cos^2(\theta/2)}$$
Since we assume that $2|\cos(\theta/2)|=|\Tr(\gamma)|\leq 2-\eta$ then $1-\cos^2(\theta/2)\geq \eta/2$, thus concluding the proof. 
\end{proof}

\subsection{Test function localized on a small spectral window}
We also recall the construction of a test function whose spherical transform is localized on a small spectral window around a large parameter $t\geq 1$. These are precisely the test function used in \cite{Kelmer10gap,FraczykHarcosMagaMilicevic24}  and we include their construction for the sake of completeness.

As in the previous section we start with a fixed non-negative smooth bi-$K$ invariant nonnegative function $f$ supported on $B_1$ and let $F(g)=f*\hat{f}$ and $h(r)=SF(\frac12+ir)$. We then let $h_t(r)=\tfrac{h(r+t)+h(r-t)}{2}$ and let $F^{s}_t=S_0^{-1}h_t$.  We recall the results of \cite[Lemma 4.3]{Kelmer10gap} for these test functions.
\begin{prop}\label{p:Ft}
Let $h_t$ and $F^s_t$ be as above. Then 
\begin{enumerate}
\item For any $r\in \bbR\cup i\bbR$ we have that  $h_t(r)\geq 0$ and  $h_t(r)\gg 1$ for $|r-t|\leq 1$.
\item At the identity we have $|F^s_t(1)|\asymp  t$
\item For any $\gamma\in \PSL_2(\bbR)$ with $\Tr(\gamma)\neq 2$ let $G_{\gamma}$ denote the centralizer of $\gamma$ in $G$. Then the orbital integral
$$\int_{G_{\gamma}\bs G}|F^s_t(g^{-1}\gamma g)|dg\leq \int_{G_{\gamma}\bs G}F(g^{-1}\gamma g)dg$$
\end{enumerate}
\end{prop}

\subsection{Test function localized on a discrete series representation.}
It follows from  (\ref{e:SPhim}) that the the spherical transform of the test function $F_m(g)=\frac{2|m|+1}{4\pi}\Phi_m(g)$, with $\Phi_m$ the spherical function defined in (\ref{e:mSpherical}), is localized precisely on the Discrete series representation $\frak{D}_m$. The following estimate is a refinement of \cite[Lemma 4.2]{Kelmer10gap} that we will need in order to improve the exponent in  \cite[Theorem 2]{Kelmer10gap} for non-spherical representations. 
\begin{prop}\label{p:PhiEst}
For $|m|\geq 2$ let $\eta \in (\frac{2\log(m)}{m},1)$.
\begin{enumerate}
\item
For any $\gamma\in \PSL_2(\bbR)$ with $|\Tr(\gamma)|\leq 2-\eta$ and any finite volume set $B \subseteq G$ we can bound 

$$\int_B|\Phi_m(g^{-1}\gamma g)|dg\leq \frac{2\log(m)}{\eta m}+\frac{\vol(B)}{m}.$$
\item For any $c>0$ there is $C\geq 1$ such that for all  $\gamma\in \PSL_2(\bbR)$ with $\|\gamma\|\geq C$ we can bound 
$$\int_{B_c}|\Phi_m(g^{-1}\gamma g)|dg\ll (C/\| \gamma\|)^{2|m|}.$$
\end{enumerate}
\end{prop}
\begin{proof}
Since $|\Tr(\gamma)|<2$ it is elliptic and there is some $\tau\in \PSL_2(\bbR)$ so that $\gamma=\tau k_\theta \tau^{-1}$ with 
$\Tr(\gamma)=2\cos(\theta/2)$. Hence the condition that $|\Tr(\gamma)|\leq 2-\eta$ implies that $\sin^2(\theta/2)\geq \frac{\eta}{2}$.
We use twisted polar coordinates $g=\tau k_1 a_t k_2$ so that $dg=\sinh(t)dtdk_1dk_2$ and 
$$g^{-1}\gamma g=k_2^{-1}a_{-t}k_1^{-1}\tau^{-1}\tau k_\theta \tau^{-1} \tau k_1 a_t k_2=k_2^{-1}a_{-t}k_\theta a_t k_2.$$
Using these coordinates we can bound 
$$|\Phi_m(g^{-1}\gamma g)| \leq \cosh(\frac{H(a_{-t}k_\theta a_t)}{2})^{-2m}=(\frac{\cosh(H(a_{-t}k_\theta a_t) +1}{2})^{-m}.$$
By Lemma \ref{l:kak} we have that 
$$\frac{\cosh(H(a_{-t}k_\theta a_t) +1}{2}=1+\sin^2(\frac{\theta}{2})\sinh^2(\frac{t}{2}).$$
Let $t_0$ satisfy that $\eta \sinh^2(t_0)= \frac{2\log(|m|)}{|m|}$ and note that for any $t\geq t_0$ we can bound 
$$|\Phi_m(g^{-1}\gamma g)| \leq (1+2\frac{\log(m)}{m})^{-m}\leq \frac{1}{m}.$$
We can now split our set $B=B^-\cup B^+$ where $B^-=\{g=\tau k_1a_tk_2\in B: t\leq t_0\}$ and bound 
$\int_{B^+}
|\Phi_m(g^{-1}\gamma g)|dg\leq \frac{\vol{B}}{m}$. When $t\leq t_0$ we use the trivial bound $|\Phi_m(g^{-1}\gamma g)|\leq 1$
to get 
$$\int_{B^{-}}|\Phi_m(g^{-1}\gamma g)|dg\leq \int_{0}^{t_0}\sinh(t)dt=\cosh(t_0)-1=2\sinh^2(\frac{t_0}{2})=\frac{2\log(m)}{\eta m}.$$

The second estimate follows from the identity 
$$|\Phi_m(g)|=\frac{2}{(1+\cosh(H(g))})^{|m|}=(\frac{4}{1+\|g\|^2})^{|m|}$$
together with the observation that  $\|g^{-1}\gamma  g\|\geq \frac{\|\gamma\|}{\|g\|^2}\gg \|\gamma\|$ uniformly for all $g\in B_c$.
\end{proof}

\subsection{Test function localized on discrete series representations}
 In order to control the contribution of non-spherical representations with values of $m$ in a large window we consider  smoothed sums of the spherical function $\Phi_m$. 
Fix a smooth non negative function $\rho$ supported on the interval $[\frac{3}{4},5/4]$ with $\rho(x)= 1$ for $x\in [1,2]$ and for a large parameter $T$ let 
\begin{equation}\label{e:F_M} 
F^{dis}_T(g)=\frac{T}{4\pi}\sum_{m\in \Z} \rho(\frac{m}{T})\Phi_m(g).
\end{equation}

The goal of this section is to prove the following estimate
\begin{prop}\label{p:FM}
The test function $F_T^{dis}$ as above satisfies the following.
\begin{enumerate}
\item Its spherical transform satisfies 
$$SF_T^{dis}(r)=\left\lbrace\begin{array}{cc} \frac{T  \rho(\frac{m}{T})}{2|m|-1}&r=\pm i(m-\tfrac12), m\in \Z\\
0 & \mbox{ otherwise}\end{array}\right.$$
\item For any $\gamma\in G$ with $|\Tr(\gamma)|>2$ we have $\int_{G_\gamma\bs G}F^{dis}_T(g^{-1}\gamma g)dg=0$.
\item For any $c>0$ there is $C\geq 1$ such that for any $\gamma\in G$ we have
$$\int_{B_c}|F^{dis}_T(g^{-1}\gamma g)|dg\ll \left\lbrace\begin{array}{lr}
1 & |\Tr(\gamma)|<2-T^{\epsilon-2}\\%  \|\gamma\|\leq C\\
T^2 &  2-T^{\epsilon-2}< |\Tr(\gamma)|<2\\%\; \|\gamma\|\leq C\\
T^2(\frac{C}{\|\gamma\|})^{3T/2} & \|\gamma\|\geq C
\end{array}\right.,
$$
\end{enumerate}
\end{prop}
Before we prove this result we need a few preliminary estimates. 
\begin{lem}\label{l:FM}
For $g=k_\alpha a_t k_\beta$ we have 
$$|F^{dis}_T(k_\alpha a_t k_\beta)|\leq \frac{T^2}{(\cosh(\frac{t}{2}))^{T}}.$$
Moreover, let $\eta\in [-\frac{1}{2},\frac{1}{2})$ so that $\eta\equiv \frac{\alpha+\beta}{2\pi}\pmod{\Z}$, then assuming  $|\eta| \geq T^{\epsilon-1}$ we can bound 
$$|F^{dis}_T(k_\alpha a_t k_\beta)| \ll_{\epsilon,\ell} \frac{1}{T^\ell \cosh(\frac{t}{2})^{T}}$$

\end{lem}
\begin{proof}
Let $\rho_T(x)=\rho(\frac{x}{T})$ so the Fourier transform $\hat{\rho}_T(y)=T\hat\rho(Ty)$ with 
$\hat\rho(z)=\int_{\bbR} \rho(x)e^{2\pi i xz} dx$. Let $\eta=\frac{(\alpha+\beta)}{2\pi}$. Recalling that $\rho(x)=0$ for $x\leq 0$, by Poisson summation 
\begin{eqnarray*}
F^{dis}_T(k_\alpha a_t k_\beta)&=&\frac{T}{4\pi} \sum_{m\in \Z} \rho_T(m)e^{2\pi im (\eta+i\frac{\log(\cosh(\frac{t}{2}))}{\pi})}\\
&=&\frac{T}{4\pi}  \sum_{m\in \Z} \hat\rho_T(m+\eta+i\frac{\log(\cosh(\frac{t}{2}))}{\pi})\\
&=&\frac{T^2}{4\pi} \sum_{m\in \Z} \hat\rho(T(m+\eta+i\frac{\log(\cosh(\frac{t}{2}))}{\pi})).\\
\end{eqnarray*}
For $z=x+iy$ with $y>0$ we can estimate using integration by parts and the fact that $\rho$ is supported on $[3/4,5/4]$
$$|\hat\rho(Tz)|\ll_{\rho,k} \frac{e^{-3\pi T y/2} }{T^k |z|^k}$$
hence 
$$T^2 |\hat\rho(T(m+\eta+i\frac{\log(\cosh(\frac{t}{2}))}{\pi}))|\ll T^{2-k}|m+\eta|^{-k} \cosh(\frac{t}{2})^{-3T/2}.$$
The contribution of all $m\neq 0$ are hence bounded by $O(T^{2-k} \cosh(\frac{t}{2})^{-3T/2})$ and  we can trivially bound contribution of $m=0$ by $O(T^2 \cosh(\frac{t}{2})^{-3T/2})$) giving the general bound. Now when  $\eta>0$ we can further bound
\begin{eqnarray*}
F^{dis}_T(k_\alpha a_t k_\beta)&\ll_k  & T^2 | \cosh(\frac{t}{2})^{-2T/2}|T \eta|^{-k}
\end{eqnarray*}
Using the assumption that $T\eta\geq T^{\epsilon}$ and taking $k$ large enough so that $\epsilon k\geq \ell+2$ concludes the proof.
\end{proof}

In order to apply this to bound $F^{dis}_T(g^{-1}\gamma  g)$ we will use the following simple observation.
\begin{lem}\label{l:kak2}
For any $t>0$ and non zero $\theta\in (-\pi ,\pi)$ decompose 
$$a_{-t}k_\theta a_t=k_\alpha a_H k_\beta$$
If  $|\theta|\geq  \delta$ and $|\theta\pm \pi|\geq  \delta$, then 
$|\alpha+\beta|\in [\delta,2\pi-\delta]$.
\end{lem}
\begin{proof}
For $\theta$ fixed let 
$$\vartheta(t)=\mathrm{Arg}(1+\sinh^2(\frac{t}{2})(1-e^{-i\theta}))$$
so that
$$\tan(\vartheta(t))=\frac{\sinh^2(\frac{t}{2})\sin(\theta)}{1+\sinh^2(\frac{t}{2})(1-\cos(\theta))}.$$
Note that $\vartheta(0)=0$ and $\vartheta(t)$ is increasing if $\theta>0$ and decreasing if $\theta<0$. Indeed, since $\tan$ is monotone it is enough to show this for $\tan(\vartheta(t))$; taking a derivative we see that 
$$\frac{d}{dt}\tan(\vartheta(t))=\frac{\sinh(\frac{t}{2})\cosh(\frac{t}{2})\sin(\theta) }{(1+\sinh^2(\frac{t}{2})(1-\cos(\theta))^2},$$
which is positive when $\theta>0$ and negative when $\theta<0$.
In the limit as $t\to \infty$ we have 
$$\lim_{t\to\infty}\tan(\vartheta(t))=\frac{\sin(\theta)}{(1-\cos(\theta))}=\frac{1+\cos(\theta)}{\sin(\theta)}.$$
In particular this shows that $\vartheta(t)\in [0,\frac{\pi}{2})$ when $\theta>0$ and in  $(-\frac{\pi}{2},0]$ when $\theta<0$.
Hence since $\alpha+\beta=\theta+2\vartheta(t)$  we get that 
$\alpha+\beta \in [\theta, \theta+\pi] \subseteq [\delta, 2\pi-\delta)$  when $\theta\in [\delta,\pi-\delta]$ and similarly 
$\alpha+\beta\in (-2\pi+\delta, -\delta]$  when $\theta\in [\delta-\pi ,-\delta]$.
\end{proof}

We can now give the proof of our main estimate.
\begin{proof}[Proof of Proposition \ref{p:FM}]
The first part follows directly from (\ref{e:SPhim}) and the second from (\ref{e:Htransform2}). For the last part we bound 
$$\int_{B_c}|F^{dis}_T(g^{-1}\gamma g)|dg\leq \vol(B_c)\sup\{|F^{dis}_T(g^{-1}\gamma g)|:g\in B_c\}.$$
When $\|\gamma\|\geq 2C$ is large then $\|g^{-1}\gamma g\|\geq C$ for all $g\in B_c$ and we can use the first part of Lemma \ref{l:FM} to conclude that $|F^{dis}_T(g^{-1}\gamma g)|\ll T^2(\frac{C}{\|\gamma\|})^T$, and similarly that $|F^{dis}_T(g^{-1}\gamma g)|\ll T^2$ when $\|\gamma\|\leq 2C$ is small. Finally, if $|\Tr(\gamma)|<2-T^{\epsilon-2}$ then this is also true for $g^{-1}\gamma g$ for any $g$.
We can thus write $g^{-1}\gamma g=\tau^{-1}k_\theta \tau$ with  $\Tr(\gamma)=2\cos(\frac{\theta}{2})$. The condition on $\Tr(\gamma)$ implies that  $\sin^2(\frac{\theta}{2})\geq \frac{1}{2}T^{\epsilon-2}$ and hence $|\theta| \in [\delta,\pi-\delta]$ with $\delta=T^{\epsilon/2-1}$.
Now write  $\tau=k_1 a_t k_2$ so that $g^{-1}\gamma g=\tau^{-1} k_\theta \tau=k_2^{-1}a_{-t}k_\theta a_tk_2$ to get that
$$F^{dis}_T(g^{-1}\gamma g)=F^{dis}_T(a_{-t}k_\theta a_t).$$
Writing $a_{-t}k_\theta a_t=k_\alpha a_H k_\beta$, by Lemma \ref{l:kak} we have that $|\alpha+\beta|\in [\delta, 2\pi-\delta)$ and hence 
$\eta=\frac{\alpha+\beta}{2\pi}\pmod{\Z}$ is also bounded below by $|\eta|\geq T^{\epsilon/2-1}$. By the second part of Lemma \ref{l:FM} (with $\ell=1$) we get that 
$$|F^{dis}_T(g)|\ll_{\epsilon} 1,$$
as claimed.
\end{proof}

\section{Proof of main Theorems }
\subsection{Proof of Theorem \ref{t:main1}}
 Given a non-tempered irreducible representation $\pi\in \hat{G}$ we have $\pi\cong \pi_1\otimes\ldots \otimes \pi_d$ with at leasts one of the factors a complementary series representation. Let $\{1,\ldots, d\}= J_1\cup J_2\cup J_3$ where %$J_0=\{j_0\}$ with $\pi_{j_0}\cong \pi_{s_0}$ is complementary with $s_0=\frac{1}{p(\pi)}$, the set 
 $J_1=\{1\leq j\leq d: \pi_j \mbox{ is complementary series or } \pi_j\cong \frak{D}_{\pm 1} \}$,\;  $J_2=\{j: \pi_j \mbox{ is principal series}\}$, and $J_3=\{1\leq j\leq d: \pi_j\cong \frak{D}_{m_j}: |m_j|\geq 2\}$. Let $j_0\in J_1$ with $\pi_{j_0}\cong \pi_{s_0}$  with $s_0=\frac{1}{p(\pi)}$, and for any $j\in J_2$ let $t_j\in \R$ so that $\pi_j\cong \pi_{\tfrac{1}{2}+it_j}$. Let $m=(m_1,\ldots, m_d)$ so that $m_j=0$ if $\pi_j$ is spherical and $\pi_j\cong \frak D_{m_j}$ otherwise.
  
 To bound the multiplicity $\m(\pi,\Gamma(\frak a))$ we use the trace formula with the test function 
$F(g)=\prod_j F^{(j)}(g_j)$, where the functions $F^{(j)}$ are defined as follows: First we set $F^{(j_0)}=F^{ex}_R$ as in Proposition \ref{p:FR} for a parameter $R$ to be determined later. For $j\in J_1\setminus\{j_0\}$ we let $F^{(j)}=f*\check f$ with 
$f\in \calF_{m_j}$ a fixed smooth non-negative function supported on $B_1$, and note that its spherical transform $h_j(r)=SF_2(r)\geq 0$ for $r\in \R\cup i\R$ and after maybe rescaling we may assume that $h_j(\frac{i}{2})=1$; For $j\in J_2$ we let 
$F^{(j)}=F^{s}_{t_j}$ as defined in Proposition \ref{p:Ft}, and for $j\in J_3$ we let $F^{(j)}=\frac{2|m_j|+1}{4\pi}\Phi_{m_j}$. 

Since the spherical transforms $h_j=SF^{(j)}$ are all  non-negative for $j\in J_1\cup J_2$ and satisfy that 
$h_{j_0}(s_0)\asymp e^{2R(1-s_0)}$ (see Proposition \ref{p:FR}), that $h_j(t_j)\asymp 1$ for $j\in J_3$ (see Proposition \ref{p:Ft}) and that $h_j(\tfrac{i}{2})=1$ for $j\in J_2$, the left hand side of the trace formula is bounded below by $\m(\pi,\Gamma(\frak a)) e^{2R(1-s_0)}$.
The right hand side of the trace formula is given by 
\begin{eqnarray*}
\vol(\Gamma(\frak a)\bs G^d)F(1)+{\sum_{\gamma\in \Gamma(\frak a)}}'\int_{\Gamma(\frak a)\bs G^d}F(g^{-1}\gamma g)dg
\end{eqnarray*}
where the sum $\Sigma'$ indicates we are only summing over $\gamma\in \Gamma(\frak a)$ with $|\Tr(\gamma_j)|<2$ for $j\in J_3$.
Fix a fundamental domain $B\subseteq G^d$ for $\Gamma\bs G$ and note that we can cover 
the fundamental domain for $\Gamma(\frak a)\bs G^d$ by a disjoint union of translates  $\tau_j B$ with $\tau_j\in \Gamma(\frak a)\bs \Gamma$, and after a change of variables we see that 
$$\int_{\Gamma(\frak a)\bs G^d}F(g^{-1}\gamma g)dg=\sum_{\tau_j\in \Gamma(\frak a)\bs \Gamma}\int_{B}F(g^{-1}\tau_j\gamma \tau_j^{-1}g)dg.$$
Recalling that $\Gamma(\frak a)$ is normal in $\Gamma$ we get that 
\begin{eqnarray*}
{\sum_{\gamma\in \Gamma(\frak a)}}'\int_{\Gamma(\frak a)\bs G^d}F(g^{-1}\gamma g)dg&=&\sum_{\tau_j\in \Gamma(\frak a)\bs \Gamma}{\sum_{\gamma\in \Gamma(\frak a)}}'\int_{B}F(g^{-1}\tau_j\gamma\tau_j^{-1} g)dg\\
&=& [\Gamma:\Gamma(\frak a)]{\sum_{\gamma\in \Gamma(\frak a)}}'\int_{B}F(g^{-1}\gamma g)dg\\
\end{eqnarray*}
Since $\Gamma$ is uniform then $B$ is compact and there is some $c>0$ such that $B\subseteq B_c^d$.  Noting that 
$V(\frak{a})=\vol(\Gamma(\frak a)\bs G^d)\asymp [\Gamma:\Gamma(\frak a)]$, and that 
$$F(1)\asymp e^R \prod_{j\in j_2}(1+|t_j|)\prod_{j\in J_3}|m_j|=e^R T(\pi),$$ comparing the two sides of the trace formula we get
$$
\m(\pi,\Gamma(\frak a)) e^{2R(1-s_0)}\ll V(\frak a) \bigg( e^RT(\pi)+{\sum_{\Gamma(\frak a)}}' \prod_j \int_{B_c}F^{(j)}(g^{-1}\gamma_j g)dg\bigg).
$$

To bound the sum over the non-trivial elements we consider lattice elements with $\|\gamma_{j_0}\|^2\in [x,2x]$ and $\|\gamma_j\|\leq C$ for all other $j\in J_1\cup J_2$. Further split $J_3=J_3'\cup J_3''$  and assume that that for any $j\in J_3'$ we have that  again $\|\gamma_j\|\leq C$ and that either $2-2^{1-k_j}< |\Tr(\gamma_j)|<2-2^{-k_j}$ with $1\leq 2^{k_j}\leq \frac{|m_j|}{\log(|m_j|)}$ or $|\Tr(\gamma_j)|>2-\frac{2\log|m_j|}{|m_j|}$. For  $j\in J_3''$ we have that $\|\gamma_j\|^2 \in [k_j,k_j+1]$ with $k_j\geq C^2$. For such an element we can bound 
$$\prod_j \int_{B_c}|F^{(j)}(g^{-1}\gamma_j g)|dg\bigg)\ll \frac{e^R}{x^{1/2}}\prod_{j\in J_3'} 2^{k_j}\prod_{j\in J_3''}k_j^{-|m_j|},$$
where in the case when $|\Tr(\gamma_j)|>2-\frac{2\log|m_j|}{|m_j|}$ with $j\in J_3'$  the factor of $2^{k_j}$ is replaced by $|m_j|$.
By Proposition  \ref{p:NGamma},  the number of such lattice points is bounded by 
$$O(\prod_{j\in J_3'} 2^{-k_j}(\frac{x^{1+\epsilon}}{V(\frak a)}+\frac{(\prod_{j\in J_3''}k_j)^\epsilon x^{1/2+\epsilon}}{V(\frak a)^{2/3}})$$ and hence the contribution of these elements is bounded by  
$$O((\prod_{j\in J_3''}k_j^{\epsilon-|m_j|} e^R (\frac{x^{1/2+\epsilon}}{V(\fraka)}+ \frac{x^{\epsilon}}{V(\fraka)^{2/3}})).$$
 Summing over all values of $k_j\leq \log(|m_j|)$ for $j\in J_3'$ and $k_j\geq C^2$ for $j\in J_3''$, and all $x\leq e^{2R}$ a power of two, and then considering all partitions of $J_3$ we get that 
$${\sum_{ \gamma\in \Gamma(\frak a)}}'\prod_j \int_{B_c}|F^{(j)}(g^{-1}\gamma_j g)|dg\bigg)\ll 
T(\pi)^\epsilon( \frac{e^{R(2+\epsilon)}}{V(\frak a)}+ \frac{e^{R(1+\epsilon)}}{V(\frak a)^{2/3}}).$$
Comparing the two sides of the trace formula  we thus get that 
$$
\m(\pi,\Gamma(\frak a)) e^{2R(1-s_0)}\ll e^RV(\frak a) T(\pi)+T(\pi)^\epsilon( e^{R(2+\epsilon)}+ e^{R(1+\epsilon)}V(\frak a)^{1/3}).
$$
Choosing $R$ so that   $e^R=T(\pi)V(\frak a)$ we get that 
$$
\m(\pi,\Gamma(\frak a))\ll (V(\frak a) T(\pi))^{2s_0+\epsilon}
$$
and since $s_0=\frac{1}{p(\pi)}$ this concludes the proof of Theorem \ref{t:main1}.\qed

\subsection{Proof of Theorem \ref{t:main2}}
Next in order to bound $\m(\frak B, \Gamma(\frak a))$, we first consider for any $T=(T_1,\ldots, T_d)$ with $T_j$ powers of $2$,  the following sets 
$$\frak{B}_T(p_0)=\left\{ \pi_1\otimes\ldots\otimes  \pi_d: p(\pi)\geq p_0 \mbox{ and }  \left\{\begin{array}{cc}T(\pi_j)\in  [T_j,2T_j) & T_j\geq 2\\
T(\pi_j)<2 & T_j<2
\end{array}\right.\right\}.$$
We then show that for any such set we have 
\begin{equation}\label{e:mB}
\m(\frak{B}_T(p_0),\Gamma(\frak a)) \ll_\epsilon  (V(\frak a)|T|^2)^{\frac{2}{p_0}+\epsilon}
\end{equation}
where $|T|=\prod_j T_j$.

As a first step we partition the set $\frak{B}_T(p_0)$ into finitely many subsets parametrized by a partitions $\calJ$ of  $\{1,\ldots, d\}=\{j_0\}\cup J_1\cup J_2,\cup J_3\cup J_4$. For such a partition consider the subset 
$$\frak B_\calJ=\left\{ \pi_1\otimes\ldots\otimes  \pi_d:   
\begin{array}{lc}
p(\pi_{j})\geq p_0 & j=j_0\\
 T(\pi_{j})< 2 & j\in J_1 \\
\pi_j \mbox{ spherical with }
T(\pi_j)\in [T_j,2T_j]& j\in J_2\\
\pi_j\cong \frakD_{m_j},\; m_j\in [T_j,2T_j]& j\in J_3\\
\pi_j\cong \frakD_{m_j},\; -m_j\in [T_j,2T_j]& j\in J_4
\end{array}\right\}$$
Given such a partition and a small parameter $\epsilon>0$ we use the trace formula with the test function $F(g)=\prod_j F^{(j)}(g_j)$ where $F^{(j)}$ defined as follows:
$F^{(j_0)}=F^{ex}_R$ as in Proposition \ref{p:FR}; for $j\in J_1$ we let 
$F^{(j)}(g)=f*\check{f}(g)+\overline{ f * \check{f}(g)}$ with $f\in \calF_1$ supported on $B_1$; for $j\in J_2$ we let 
$F^{(j)}=F^{sp}_{T_j}$ as in Proposition \ref{p:FT2};  for $j\in J_3$ we let $F^{(j)}= F^{dis}_{T_j}$  and for $J\in J_4$ we have $F^{(j)}=\overline{ F^{dis}_{T_j}}$ where $F^{dis}_{T_j}$ is as in Proposition \ref{p:FM}.

Using linearity of the trace formula, the left hand side can be bounded below by 
$$\m(\frak{B}_\calJ,\Gamma(\frak a))e^{2R(1-\frac{1}{p_0})},$$
and the right hand side is bounded above by a constant multiple of
$$V(\frak a)\bigg(|T|^2 e^R+{\sum_{\gamma\in \Gamma(\frak a)}}' \prod_j \int_{B_c}|F^{(j)}(g^{-1}\gamma_j g)|dg\bigg),$$
where $c>0$ is sufficiently large so that a fundamental domain for  $\Gamma\bs G^d$ is contained in $B_c^d$ and the sum $\Sigma'$ is over all non trivial lattice elements in $\Gamma(\frak a)$  with $|\Tr(\gamma_j)|<2$ for all $j\in J_3\cup J_4$.  By Proposition 
\ref{p:FT2}, we may further restrict the sum to elements with $\|\gamma_j\|\leq C$ for all $j\in J_1\cup J_2$ and with $|\Tr(\gamma_j)|<2+\frac{1}{T_j}$ for $j\in J_2$.
Now further split $J_2=J_2'\cup J_2''$ and $J_3\cup J_4=J'\cup J''\cup J^+$. Fix $x\leq (Ce^R)^2$ some power of $2$, for any $j\in J_2''$ fix an integer $k_j\in [0,\log_2(T_j)]$, for $j\in J^+$ fix and integer $k_j\geq 4C^2$, and consider elements $\gamma\in \Gamma(\frak a)$ satisfying the following conditions: 
\begin{enumerate}
\item  $\|\gamma_{j_0}\|^2\in [x,2x]$.
\item  $\|\gamma_j\|\leq2C$ and $|\Tr(\gamma_j)|-2|<\frac{4}{T^2_j}$ for $j\in J_2'$,
\item  $\|\gamma_j\|\leq 2C$ and  $|\Tr(\gamma_j)|\in [2-2^{-k_j-1},2-2^{-k_j}]$  for $j\in J_2''$,
\item $\|\gamma_j\|\leq 2C$ and $|\Tr(\gamma_j)|>2-\frac{1}{T_j^{2-\epsilon}}$  for $j\in J'$,
\item  $\|\gamma_j\|\leq 2C$ and  $|\Tr(\gamma_j)|<2-\frac{1}{T_j^{2-\epsilon}}$ for $j\in J''$,
\item $\|\gamma_j\|^2\in [k_j,k_j+1]$ for $j\in J^+$.
\end{enumerate}
For any such $\gamma\in \Gamma(\frak a)$ using the estimates in Propositions \ref{p:FR}, \ref{p:FT2} and \ref{p:FM} we can bound 
$$\int_{B_c}|F^{(j)}(g^{-1}\gamma_j g)|dg\ll_\epsilon\left\{ \begin{array}{cc}
\frac{e^R}{\sqrt{x}}& j=j_0\\
1& j\in J_1\\
T_j^2 & j\in J_2'\\
2^{k_j} & j\in J_2''\\
T_j^2 & j\in J'\\
1 & j\in J''\\
T_j^2(C/k_j)^{3T_j/2} & j\in J^+
\end{array}\right.$$
so that
$$\prod_j \int_{B_c}|F^{(j)}(g^{-1}\gamma_j g)|dg\ll \frac{e^R}{x^{1/2}} \prod_{j\in J_2''}2^{k_j}\prod_{j\in J_2 '} T_j^2 \prod_{j\in J_3'} T_j^2 \prod_{j\in J_3^+} T_j^2 (\frac{C}{k_j})^{3T_j/2}.$$
Next we point out that by Proposition \ref{p:NGamma} the number of such elements in $\Gamma(\frak a)$ is bounded by 
$$O\big(\prod_{j\in J_2'}\frac{1}{T_j^2}\prod_{j\in J_2''} 2^{-k_j}\prod_{j\in J'}\frac{1}{T_j^{2-\epsilon}}\big( \frac{x^{1+\epsilon}}{V(\frak a)}+\prod_{j\in J^+}k_j^\epsilon \frac{ x^{1/2+\epsilon}}{V(\frak a)^{2/3}})\big),$$
so that the contribution of all these elements to the sum is bounded by 
$$O\bigg(e^R\prod_{j\in J'} T_j^\epsilon \prod_{j\in J^+} T_j^2(\frac{C}{k_j})^{3T_j/2}(\frac{x^{1/2+\epsilon}}{V(\frak a)}+\prod_{j\in J^+} k_j^\epsilon \frac{x^{\epsilon}}{V(\frak a)^{2/3}})  \bigg).$$
Next, we sum over all $k_j$  noting that for each $j\in J^+$ the sum over all $k_j\geq 2C$ 
$$\sum_{k_j\geq 2C} T_j^2(\frac{C}{k_j})^{3T_j/2}\ll_C T_j2^{-3T_j/2}\ll 1,$$
converges and is uniformly bounded and that for $j\in J_2'$ we have  $ \log_2(T_j)\ll_\epsilon T_j^\epsilon$ values of $k_j$ to sum over. So the contribution of all possible values of $k_j$ is bounded by 
$$O(|T|^\epsilon e^R(\frac{x^{1/2+\epsilon}}{V(\frak a)}+\frac{x^{\epsilon}}{V(\frak a)^{2/3}}),$$
where $|T|=\prod_j T_j$. Finally summing over all possible partitions of $J_2$ and $J_3\cup J_4$ and over all values of $x\ll e^{2R}$ a power of $2$ we get that 
 $${\sum_{\gamma\in \Gamma(\frak a)}}' \prod_j \int_{B_c}|F^{(j)}(g^{-1}\gamma_j g)|dg)\ll (|T|^\epsilon (\frac{e^{R(2+\epsilon)}}{V(\frak a)}+\frac{e^{R(1+\epsilon)}}{V(\frak a)^{2/3}}),$$
 and hence 
 $$\m(\frak{B}_\calJ,\Gamma(\frak a))e^{2R(1-\frac{1}{p_0})}\ll V(\frak a) |T|^2 e^R +|T|^\epsilon (e^{R(2+\epsilon)}+ e^{R(1+\epsilon)}V(\frak a)^{1/3}).$$
Taking the value of $R$ so that $e^R=V(\frak a)|T|^2$ we get that 
$$\m(\frak{B}_\calJ,\Gamma(\frak a))\ll_\epsilon (V(\frak a)|T|^2)^{\frac{2}{p_0}+\epsilon},$$
and summing over the all partitions $\calJ$ of $\{1,\ldots d\}$ gives  (\ref{e:mB}).

Finally, to deal with a general set $\frak B$ note that $\frak{B}$ is contained in the union of the sets $\frak{B}_T(p_0)$ with $p_0=p(\frak B)$ 
and $T=(2^{k_1},\ldots, 2^{k_d})$ with $k_1,\ldots, k_d\in \Z^+$ with $\sum_j k_j\leq \log_2(T(\frak B))$. Since there are at most 
$\log(T(\frak B))^d\ll_\epsilon T(\frak B)^\epsilon$ such choices of $T$ we get that 
$$\m(\frak{B})\ll_\epsilon (V(\frak a)T(\frak B)^2)^{\frac{2}{p(\frak B)}+\epsilon},$$
concluding the proof of Theorem \ref{t:main2}.\qed

%----------------------------------------------------------------
%GATHER{Mybib.bib}   % For Gather Purpose Only
%\bibliographystyle{amsplain}
%\bibliography{Mybib}

\begin{thebibliography}{10}

\bibitem{Efrat87}
I.~Efrat, \emph{The Selberg trace formula for $\mathrm{PSL}_2(\mathbb{R})^n$},
  Mem. Amer. Math. Soc. \textbf{65}.

\bibitem{FraczykGorodnikNevo24}
M.~Fraczyk, A.~Gorodnik, and A.~Nevo, \emph{Automorphic density
  estimates and optimal diophantine exponents}, 2024.

  
  \bibitem{FraczykHarcosMagaMilicevic24}
M.~Fraczyk, G.~Harcos, P.~Maga, and D.~Mili\'cevi\'c,
  \emph{The density hypothesis for horizontal families of lattices}, 
  Amer. J. Math. \textbf{146} (2024), no.~1, 107--160

   
\bibitem{GolubevKamber23}
K.~Golubev and A.~Kamber, \emph{On {S}arnak's density conjecture
  and its applications}, Forum Math. Sigma \textbf{11} (2023), Paper No. e48,
  51. \MR{4603107}

\bibitem{Hejhal76}
D.~A. Hejhal, \emph{The {S}elberg trace formula for {${\rm PSL}(2,R)$}.
  {V}ol. {I}}, Springer-Verlag, Berlin, 1976, Lecture Notes in Mathematics,
  Vol. 548.

%\bibitem{kelmer10}
%Dubi Kelmer, \emph{Distribution of holonomy about closed geodesics in a product
%  of hyperbolic planes}, 2010.

\bibitem{Kelmer10gap}
D.~Kelmer, \emph{A uniform strong spectral gap for congruence covers of a
  compact quotient of $\PSL(2,\bbR)^d$}, Int. Math. Res. Not. (2010).

\bibitem{KelmerSarnak09}
D.~Kelmer and Peter Sarnak, \emph{Strong spectral gaps for compact quotients
  of products of $\PSL(2,\bbR)$}, J. Eur. Math. Soc. \textbf{11} (2009), no.~2,
  283--313.
  
  \bibitem{KimSarnak03}
H.~H. Kim and P.~ Sarnak, \emph{Refined estimates towards the
Ramanujan  and Selberg conjectures}, Appendix to Henry~H. Kim, J. Amer. Math. Soc.
  \textbf{16} (2003), no.~1, 139--183

\bibitem{Lang85}
S.~Lang, \emph{{${\rm SL}\sb 2({\bf \bbR})$}}, Graduate Texts in Mathematics,
  vol. 105, Springer-Verlag, New York, 1985, Reprint of the 1975 edition.

\bibitem{Margulis91}
G.~A. Margulis, \emph{Discrete subgroups of semisimple {L}ie groups},
  Ergebnisse der Mathematik und ihrer Grenzgebiete (3) [Results in Mathematics
  and Related Areas (3)], vol.~17, Springer-Verlag, Berlin, 1991. \MR{MR1090825
  (92h:22021)}

\bibitem{SarnakXue91}
P.~Sarnak and X.~Xi Xue, \emph{Bounds for multiplicities of automorphic
  representations}, Duke Math. J. \textbf{64} (1991), no.~1, 207--227.

\bibitem{Selberg95}
A.~Selberg, \emph{Partial zeta function}, Mittag-Leffler Inst. lecture notes,
  1995.

\bibitem{Weil60}
A.~Weil, \emph{Algebras with involutions and the classical groups}, J.
  Indian Math. Soc. (N.S.) \textbf{24} (1960), 589--623 (1961).

\end{thebibliography}
\def\cprime{$'$} \def\cprime{$'$}
\providecommand{\bysame}{\leavevmode\hbox to3em{\hrulefill}\thinspace}
\providecommand{\MR}{\relax\ifhmode\unskip\space\fi MR }
\providecommand{\MRhref}[2]{%
  \href{http://www.ams.org/mathscinet-getitem?mr=#1}{#2}
}
\providecommand{\href}[2]{#2}

\end{document}